\documentclass{amsart}

\usepackage{amssymb}

\newcommand{\rest}{\upharpoonright}
\newcommand{\ran}{\text{ran}}
\newcommand{\eran}{\text{ran}^*\text{ }}
\newcommand{\id}{\text{id }}
\newcommand{\D}{\mathcal{D}}
\newcommand{\tri}{\triangleleft}
\newcommand{\cf}{\text{cf }}
\newcommand{\cl}{\textit{cl}}
\newcommand{\LS}{\mathrm{LS}}
\newcommand{\Ult}{\mathop{\mathrm{Ult}}}
\newcommand{\gS}{\text{gS}}

\newcommand{\ceil}[1]{\left \lceil #1\right \rceil}
\newcommand{\floor}[1]{\left \lfloor #1\right \rfloor}
\newcommand{\card}[1]{\left \vert #1\right \vert}
\newcommand{\seq}[1]{\langle #1 \rangle}

\newtheorem{theorem}{Theorem}[section]
\newtheorem{claim}[theorem]{Claim}
\newtheorem{definition}[theorem]{Definition}
\newtheorem{corollary}[theorem]{Corollary}

\newtheorem{lemma}[theorem]{Lemma}
\newtheorem{remark}[theorem]{Remark}
\newtheorem{fact}[theorem]{Fact}
\newtheorem{question}[theorem]{Question}

\newtheorem*{theorem*}{Theorem}

\title{Large cardinal axioms from tameness in AECs}
\author{Will Boney}
\email{wboney@math.harvard.edu}
\address{Mathematics Department\\Harvard University\\Cambridge, MA, USA}

\author{Spencer Unger}
\email{sunger@math.ucla.edu}
\address{Department of Mathematics\\University of California-Los Angeles\\Los Angeles, CA, USA}

\date{\today}

\begin{document}

\begin{abstract} We show that various tameness assertions about abstract
elementary classes imply the existence of large cardinals under mild cardinal
arithmetic assumptions.  For instance, we show that if $\kappa$ is an
uncountable cardinal such that $\mu^\omega < \kappa$ for every $\mu< \kappa$ and
every AEC with L{\"o}wenheim-Skolem number less than $\kappa$ is $<\kappa$-tame,
then $\kappa$ is almost strongly compact.  This is done by isolating a class of
AECs that exhibits tameness exactly when sufficiently complete
ultrafilters exist.  \end{abstract}

\maketitle

\section{Introduction}

The birth of modern model theory is often said to be Morley's proof
\cite{morleycat} of what was then called the \L o\'{s} Conjecture.  This is now
called Morley's Categoricity Theorem.  It is only natural that this same
question be an important test question when studying nonelementary model theory.
In one of the most popular contexts for this study, Abstract Elementary Classes,
this question is known as Shelah's Categoricity Conjecture.

While still open, there are many partial results towards this conjecture that
add various model-theoretic and set-theoretic assumptions.  The most relevant
for this discussion is the first author's \cite[Theorem 7.5]{tamelc}, which
shows that if there are class many strongly compact cardinals, then any
Abstract Elementary Class (AEC) that is categorical in \emph{some} high enough
successor cardinal is categorical in every high enough cardinal.  One of the
central concepts in the proof is the notion of tameness, which says roughly that
if two types differ, then they differ over some small subset of their domain.
Types here do not have the syntactic form familiar from first order-logic, since
AECs lack syntax.  Instead, a semantic version of type (called Galois or orbital
type) is introduced as, roughly, the orbit of elements under automorphisms of a
sufficiently homogeneous (or monster) model fixing the domain.  In practice
tameness has two cardinal parameters, the size of the domain of the types and
the cardinal measuring how small the subset of the domain must be.

A key instance of the advances in the first author's work is the following,
which with a little more work allows the application of previous results of
Shelah \cite{sh394} and Grossberg and VanDieren \cite{tamenessthree} to obtain a
version of Shelah's categoricity conjecture.

\begin{fact}[\cite{tamelc}.4.5] \label{boneythm}
If $\mathbb{K}$ is an AEC with $LS(\mathbb{K}) < \kappa$ and $\kappa$ is strongly compact, then $\mathbb{K}$ is $<\kappa$-tame.
\end{fact}

Sections 5 and 6 of \cite{tamelc} give similar theorems for measurable and
weakly compact cardinals.  The main theorems of this paper give converses to
these results under mild cardinal arithmetic assumptions.  We state the
following theorem as a sample application of our methods.  We prove below that
by strengthening the tameness hypothesis we can drop the ``almost" from the
conclusion of the theorem.

\begin{theorem*}
Let $\kappa$ be uncountable such that $\mu^\omega < \kappa$ for every $\mu < \kappa$.
\begin{enumerate}
	\item If $\kappa^{<\kappa} = \kappa$ and every AEC with L{\"o}wenheim-Skolem number less than $\kappa$ is $(<\kappa, \kappa)$-tame, then $\kappa$ is almost weakly compact.
	\item If every AEC with L{\"o}wenheim-Skolem number less than $\kappa$ is $\kappa$-local, then $\kappa$ is almost measurable.
	\item If every AEC with L{\"o}wenheim-Skolem number less than $\kappa$ is $<\kappa$-tame, then $\kappa$ is almost strongly compact.

\end{enumerate}
\end{theorem*}

The first step in this direction is Shelah \cite{shelah932}, where the
measurable version appears as Theorem 1.3.  The example constructed in
this paper is a generalization of Shelah's.  Note that
Shelah's proof is essentially correct, but requires minor correction (see Remark
\ref{shelahfix} for a discussion).

The proof of the main theorems all follow the same plan, which we outline here.
First, Section \ref{phase2} codes large cardinals into a combinatorial statement
$\#(\mathcal{D}, \mathcal{F})$ (see Definition \ref{hashtagdef}).  Then Section
\ref{phase1} defines two structures $H_{1}$ and $H_{2}$ such that corresponding
small substructures of them are isomorphic, but $H_{1}$ and $H_2$ are only
isomorphic if the relevant $\#(\mathcal{D}, \mathcal{F})$ holds.  Finally,
Section \ref{phase3} defines an AEC $\mathbb{K}_\sigma$ that contains $H_1$ and
$H_2$ and codes their isomorphism (and the isomorphism of their substructures)
into equality of Galois types.  This forms the connection between large
cardinals and equality of Galois types.

Our work has an immediate application to category theory.  Makkai and Par\'{e} proved a theorem \cite{makkaipare} about the accessibility of powerful images from the assumption of class many strongly compact cardinals and Lieberman and Rosicky \cite{liebermanrosicky} later applied this to AECs to give an alternate proof of Fact \ref{boneythm} above.  Brooke-Taylor and Rosicky \cite{brooketaylorrosicky} have recently weakened the hypotheses of Makkai and Par\'{e}'s result to almost strongly compact and our result completes the circle and shows that the conclusion of Makkai and Par\'{e}'s result is actually a large cardinal statement in disguise.  See Corollary \ref{ct-cor} and the surrounding discussion.

Turning back to model theory, this shows that any attempt to prove that all AECs
(even with the extra assumption of amalgamation) are eventually tame as a strategy to prove Shelah's Categoricity Conjecture will fail in ZFC.  However, the AECs
constructed in this paper are unstable and don't really fit into the picture of classification
theory so far or the categoricity conjecture.  This leaves open the possibility
that eventual tameness can be proven in ZFC from \emph{model-theoretic}
assumptions, such as stability or categoricity.  Partial work towards this goal
has already been done by Shelah \cite{sh394}, which derives a variant of
tameness from categoricity; see \cite[Theorem 11.15]{baldwinbook} for an
exposition.  A related question of Grossberg asks if amalgamation can be derived
from categoricity.

In Section \ref{strength}, we prove that even without our cardinal arithmetic
assumption we can derive large cardinal strength from tameness assertions.
Roughly speaking we show that if $\kappa$ carries the tameness property
corresponding to weak compactness, then $\kappa$ is weakly compact in $L$.

The reader is advised to have some background in both set theory and model
theory.  The set-theoretic background is in large cardinals for which we
recommend Kanamori's book \cite{kanamori}.  For the model-theoretic background,
see a standard reference on AECs such as Baldwin's book \cite{baldwinbook}.  We would like to thank John Baldwin, Andrew Brooke-Taylor, and the anonymous referee for helpful comments on this paper.

\section{Large cardinals}\label{phase2}

We begin by recalling some relevant large cardinal definitions.  These are
slight tweaks on standard definitions in the spirit of $\aleph_1$-strongly compact cardinals
(see for example \cite{BM}).  The basic framework is to take a large cardinal property that has $\kappa$ being large if there is a $\kappa$-complete object of some type and parameterizing the completeness by some $\delta$.  Then $\kappa$ is ``almost large'' if a $\delta$-complete object exists for all $\delta < \kappa$ rather than at $\kappa$.

\begin{definition} Let $\kappa$ be an uncountable cardinal.
\begin{enumerate}

\item \begin{enumerate} \item $\kappa$ is $\delta$-weakly compact if for every field $\mathcal{A} \subset \mathcal{P}(\kappa)$ 
of size $\kappa$ there is a nonprincipal $\delta$-complete
uniform filter measuring each set in $\mathcal{A}$.

\item $\kappa$ is almost weakly compact if it is $\delta$-weakly compact for all
$\delta <\kappa$.

\item $\kappa$ is weakly compact if it is $\kappa$-weakly compact.
\end{enumerate}

\item \begin{enumerate}\item $\kappa$ is $\delta$-measurable if there is a uniform, $\delta$-complete ultrafilter on $\kappa$.

\item $\kappa$ is almost measurable if it is $\delta$-measurable for all
$\delta<\kappa$.

\item $\kappa$ is measurable if it is $\kappa$-measurable.
\end{enumerate}

\item \begin{enumerate}\item $\kappa$ is $(\delta, \lambda)$-strongly compact for $\delta \leq \kappa \leq
\lambda$ if there is a $\delta$-complete, fine ultrafilter on
$\mathcal{P}_\kappa\lambda$.

\item $\kappa$ is $(\delta, \infty)$-strongly compact if it is $(\delta,\lambda)$-strongly compact for all $\lambda \geq \kappa$.

\item $\kappa$ is $\lambda$-strongly compact if it is $(\kappa, \lambda)$-strongly compact.

\item $\kappa$ is almost strongly compact if it is $(\delta, \infty)$-strongly
compact for all $\delta < \kappa$.

\item $\kappa$ is strongly compact if it is $(\kappa, \infty)$-strongly compact.
\end{enumerate}

\end{enumerate}
\end{definition}

We note that the notions of $\delta$-weakly compact and $\delta$-measurable are
not standard.  From the definitions, it can be seen that being almost measurable implies being a limit of measurables.  For almost weak and almost strong compactness, the relation is not so clear.  For instance, the following seems open.

\begin{question}
Is ``there exists a proper class of almost strongly compact cardinals''
equiconsistent with ``there exists a proper class of strongly compact
cardinals?'' \end{question}

For the section, we fix an upward directed partial ordering $(\D,\tri)$ with
$\tri$ strict.  The intended applications are $(\kappa,\in)$ and
$(\mathcal{P}_\kappa\lambda,\subset)$.

\begin{definition} For $d \in \D$ we define $\ceil{d} = \{ d' \in \D \mid d'
\tri d \}$ and $\floor{d} = \{d' \in \D \mid d \tri d' \}$. \end{definition}

We also fix a collection $\mathcal{F}$ of functions each of which has domain
$\D$.  For $f_1,f_2 \in \mathcal{F}$ we set $f_1 \leq f_2$ if and only if there
is an $e:\ran(f_2) \to \ran(f_1)$ such that $f_1 = e \circ f_2$.  Note the
witnessing $e$ is unique.  Obviously for each $f \in \mathcal{F}$, the set
$\{f^{-1}\{i\} \mid i \in \ran(f) \}$ partitions $\D$.  So $f_1 \leq f_2$ is
equivalent to saying that the partition from $f_2$ refines the partition from
$f_1$.  We require that $\mathcal{F}$ is upward directed under $\leq$.

We are interested in elements that appear cofinally often as values of $f$ so we
define
\[\ran^*(f) = \bigcap_{d \in \D} \ran(f \rest \floor{d}). \]

With this notation in mind we formulate the following principle, which is
implicit in \cite{shelah932}.

\begin{definition} \label{hashtagdef} Suppose $\mathcal{F}$ is a directed family
of functions with domain $\D$.  Let $\#(\D,\mathcal{F})$ be the assertion that
there are $f^* \in \mathcal{F}$ and a collection $\{u_f \subseteq \ran^*(f) \mid
f \in \mathcal{F} \wedge f \geq f^* \}$ of nonempty finite sets such that if $e$
witnesses that $f \geq f^*$, $e \rest u_f: u_f \to u_{f^*}$ is a bijection.
\end{definition}

Note that $e \rest u_f$ is unique, since there is a unique $e$ witnessing $f
\geq f^*$.  

This principle allows us to define a filter on $\D$.  Assume that $\#(\D,
\mathcal{F})$ holds; this assumption is active until Corollary \ref{cor-wc}.
Then we can choose $i_f \in u_f$ such that $e(i_f) = \min u_{f^*}$ where $e$
witnesses $f^* \leq f$.  Then we define $U \subseteq P(\D)$ by $A \in U$ if and
only if there are $d \in \D$ and $f \in \mathcal{F}$ with $f \geq f^*$ such that
$f^{-1}\{i_f\} \cap \floor{d} \subseteq A$.  Note that $U$ depends on the many
parameters we have defined so far: $\D, \mathcal{F}, \{u_f\}$, and $i_f$.  Also,
the choice of $i_{f^*}$ as the minimum of $u_{f^*}$ was arbitrary, any element
would have done.  Indeed, different elements generate different ultrafilters.

\begin{remark}\label{partfilter-rem}
The formulation of $\#(\D, \mathcal{F})$ given above is chosen because it is the easiest to work with in general.  However, there is an alternate formulation in terms of the partitions of $\D$ generated by the functions of $\mathcal{F}$ that can make the definition of the filter more clear.  In that language, $\#(\D,
\mathcal{F})$ holds if and only if there is a special partition $\mathcal{P}^*$ such that any finer partition $\mathcal{P}$ has a distinguished piece $X_\mathcal{P}$ that is chosen in a coherent way: if $\mathcal{Q}$ is finer than $\mathcal{P}$, then $X_\mathcal{Q} \subseteq X_\mathcal{P}$.  We also require that for $d \in \D$ and $\mathcal{P}$ finer than $\mathcal{P}^*$, we have $X_{\mathcal{P}} \cap \floor{d} \neq \emptyset$.

Then we can define the filter as follows, given $A \subset \D$, we form a partition $\mathcal{P}_A$ that is finer than both $\mathcal{P}^*$ and $\{A, \D - A\}$.  Then we set $A \in U$ if and only if the distinguished piece $X_{\mathcal{P}_A}$ is a subset of $A$ rather than $\D - A$.

The choice of $i_f \in u_f$ corresponds to choices of different distinguished pieces, showing that there are $|u_{f^*}|$-many filters with the desired property.
\end{remark}

\begin{claim}\label{filter} $U$ is a proper filter and for all $d \in \D$,
$\floor{d} \in U$.  \end{claim}

\begin{proof}  It is not hard to see that $f^{-1}\{i_f\} \cap \floor{d}$ is
nonempty for all $d$ and $f$, so $\emptyset \notin U$ provided that it forms a
filter.  The fact that $\floor{d} \in U$ for all $d$ is immediate from the
definition.

To see that $U$ is a filter, let $A,B \in U$ witnessed by $f_1,d_1$ and
$f_2,d_2$ respectively.  Let $d_3 \in \D$ be above $d_1$ and $d_2$ and $f \geq
f_1,f_2$.  This is possible since both $\mathcal{D}$ and $\mathcal{F}$ are
directed.  It follows that $f^{-1}\{i_f\} \cap \floor{d_3} \subseteq A \cap
B$.\end{proof}

We would like to generate highly complete filters.  To do so we use the
following ad-hoc definition, which is essentially a closure property of the set
of functions $\mathcal{F}$.

\begin{definition} We say that $\mathcal{F}$ is $\tau$\emph{-replete} if for
every $\mu<\tau$ and sequence $\langle B_\epsilon \mid \epsilon < \mu \rangle$
of subsets of $\D$ such that for each $\epsilon$ there is a function
$f_\epsilon$ such that $B_\epsilon=f_\epsilon^{-1}\{i\}$ for some $i$, there is
a function $f \in \mathcal{F}$ and $\{i_\alpha : \alpha < \mu\} \subset \ran f$
such that $f^{-1}\{i_0\} = \bigcap_{\epsilon<\mu} B_\epsilon$ and for all
$\alpha< \mu$ and $d \in \D$, $f(d)= i_{\alpha+1}$ if and only if $d \notin
\bigcap_{\epsilon<\mu}B_\epsilon$ and $\alpha$ is least such that $d \notin
B_\alpha$. \end{definition}

\begin{claim}\label{complete} If $(\D,\tri)$ is $\tau$-directed and
$(\mathcal{F},\leq)$ is $\tau$-replete, then $U$ is $\tau$-complete. \end{claim}

By $\tau$-directed we mean that sets of size less than $\tau$ have an
upperbound.

\begin{proof} Let $A_\epsilon$ for $\epsilon<\mu$ be elements of $U$ where $\mu
< \tau$.  By the definition of $U$, for each $\epsilon<\mu$ we have
$f_\epsilon$ and $d_\epsilon$ so that $f_\epsilon^{-1}\{i_{f_\epsilon}\} \cap
\floor{d_\epsilon} \subseteq A_\epsilon$.  Let $B_\epsilon =
f_\epsilon^{-1}\{i_{f_\epsilon}\}$ for $\epsilon < \mu$ and use the
$\tau$-repleteness of $\mathcal{F}$ to find $f$.  Using the directedness of
$\mathcal{F}$ we can find $\hat{f} \geq f,f^*$ (recall $f^*$ is given by $\#(\D, \mathcal{F})$).  Using the $\tau$-directedness
of $\D$, let $d$ be above each $d_\epsilon$ for $\epsilon<\mu$.

Let $e$ witness that $f \leq \hat{f}$, ie $f = e \circ \hat{f}$.  We want to
show that $e(i_{\hat{f}}) = 0$, since then $\hat{f}^{-1}\{i_{\hat{f}}\} \cap
\floor{d} \subseteq \bigcap_{\epsilon<\mu} (B_\epsilon \cap \floor{d_\epsilon})
\subseteq \bigcap_{\epsilon<\mu}A_\epsilon$.

Suppose that $e(i_{\hat{f}})=\epsilon$ is not zero.  Then
$\hat{f}^{-1}\{i_{\hat{f}}\} \subseteq \D - B_\epsilon$ by the definition of
$f$.  This contradicts that $U$ is filter containing all the sets $\floor{d}$
for $d \in \D$. \end{proof}

\begin{claim}\label{measure} If $A \subseteq \D$ and there is an $f$
in $\mathcal{F}$ such that $A = f^{-1}X$ for some $X
\subseteq \ran(f)$, then $U$ measures $A$. \end{claim}

\begin{proof} Let $f$ and $X$ witness the hypotheses of the claim.  Since
$\mathcal{F}$ is directed, we can find $\hat{f} \in \mathcal{F}$ such that
$f,f^* \leq \hat{f}$.  Let $e$ be such that $f = e \circ \hat{f}$.  Now it is
not hard to see that if $e(i_{\hat{f}}) \in X$, then
$\hat{f}^{-1}\{i_{\hat{f}}\} \subseteq A$ and if $e(i_{\hat{f}}) \notin X$, then
$\hat{f}^{-1}\{i_{\hat{f}}\} \subseteq \D -A$.  In the first instance
we have $A \in U$ and in the second we have $\D-A \in U$. \end{proof}

If $\mathcal{F}$ satisfies the hypothesis of the previous claim for $A$, then we
say that $\mathcal{F}$ \emph{has a characteristic function for} $A$.

We can now reformulate many large cardinal notions that are witnessed by the existence of measures.  Our first corollary is an equivalent formulation of weak compactness.

\begin{corollary}\label{cor-wc}  Let $\kappa$ be a regular cardinal.  $\kappa$
is weakly compact if and only if for all fields $\mathcal{A}$ of subsets of
$\kappa$ with $\card{\mathcal{A}} = \kappa$, $\#(\kappa,\mathcal{F})$ holds for
some set of functions $\mathcal{F}$ on $\kappa$ which is directed,
$\kappa$-replete and contains characteristic functions for all elements of
$\mathcal{A}$.  \end{corollary}

\begin{proof}  Assume that $\kappa$ is weakly
compact.  Let $\mathcal{A}$ be a field of subsets of $\kappa$ with
$\card{\mathcal{A}}= \kappa$.  We can assume that $\mathcal{A}$ is closed under
intersections of size less than $\kappa$.  Using the weak compactness of
$\kappa$, we fix a $\kappa$-complete $\mathcal{A}$-ultrafilter $U$.

Let $\mathcal{F}_\mathcal{A}$ be the collection of functions $f: \kappa \to
\kappa$ such that $\ran(f) \subseteq \alpha <\kappa$ for some $\alpha$ and for
all $\beta \in \ran(f)$, $f^{-1}\{\beta\} \in \mathcal{A}$.  It is not
difficult to show that $\mathcal{F}_\mathcal{A}$ is $\leq$-directed and it is
$\kappa$-replete since $U$ is $\kappa$-complete.

For each $f \in \mathcal{F}_{\mathcal{A}}$ let $u_f = \{i_f\}$ where $i_f$ is
the unique element of $\ran{f}$ such that $f^{-1}\{i_f\} \in U$.  If we take
$f^*$ to be the constantly zero function, then it is straightforward to see that
$f^*$ and $\{u_f \mid f \in \mathcal{F}_\mathcal{A} \}$ satisfy
$\#(\D,\mathcal{F}_\mathcal{A})$.

For the reverse direction, for each field $\mathcal{A}$ we apply Claims
\ref{filter}, \ref{complete} and \ref{measure} to see that the filter $U$
generated by $\#(\D,\mathcal{F})$ is a nonprincipal $\kappa$-complete
$\mathcal{A}$-ultrafilter. \end{proof}

\begin{remark}\label{rem-almost-wc}
A similar proof characterizes $\sigma^+$-weak
compactness where we just replace $\kappa$-replete with $\sigma^+$-replete and
consider functions with codomain $\sigma$.
\end{remark}

We also have characterizations of $\sigma^+$-measurable,
$(\delta,\lambda)$-strongly compact and $\lambda$-strongly compact.

\begin{corollary}\label{cor-almost-meas} $\kappa$ is $\sigma^+$-measurable if
and only if $\#(\kappa,\mathcal{F})$ holds for some set $\mathcal{F}$ of
functions from $\kappa$ to $\sigma$ such that $\mathcal{F}$ is directed,
$\sigma^+$-replete and has characteristic functions for all subsets of $\kappa$.
\end{corollary}

\begin{corollary}\label{cor-sc} Let $\kappa \leq \lambda$ be cardinals.
$\kappa$ is $\lambda$-strongly compact if and only if
$\#(\mathcal{P}_\kappa(\lambda),\mathcal{F})$ holds for some set
$\mathcal{F}$ of functions with domain $\mathcal{P}_\kappa(\lambda)$ and range
bounded in $\kappa$, such that $\mathcal{F}$ is directed, $\kappa$-replete and
has characteristic functions for all subsets of $\mathcal{P}_\kappa(\lambda)$.
\end{corollary}

\begin{remark}\label{rem-almost-sc} The previous corollary can be modified to
give a natural characterization of ``$\kappa$ is $(\delta,\lambda)$-strongly
compact".  \end{remark}

The proofs of these corollaries are all similar to the proof of Corollary
\ref{cor-wc} and will be omitted.

\section{Model constructions}\label{phase1}

In this section we describe a family of constructions of models which take $\D$
and $\mathcal{F}$ from the previous section as parameters.  We will also use a countable closure hypothesis on $\mathcal{F}$, but we delay this specification until Lemma \ref{tame-like}.

We define the languages and structures that are the key objects in this section and the next.  

\begin{definition}\label{language-def}
Fix a $\mathcal{D}$ and $\mathcal{F}$ as in the previous section.  
\begin{enumerate}
	\item Set $X := \cup \{\ran f : f \in \mathcal{F}\}$, $\sigma := \vert X
\vert$ and  $(G, +) := \left( [X]^{<\omega}, \Delta\right)$.
	\item $\mathcal{L}_\sigma^-$ is the language with two sorts $A$ and $I$; functions $\pi:A \to I$ and $F_c:A \to A$; and relations $P, D_c \subset A$ and $E', E, R \subset A^2$ as $c$ ranges over $G$.
	\item $\mathcal{L}_\sigma$ is the language $\mathcal{L}_\sigma^-$ with an additional sort $J$ and a function $Q:A \to J$.
\end{enumerate}
\end{definition}

The sorts here are disjoint.  $\Delta$
is the symmetric difference on finite subsets of $X$.  Note that $(G,+)$ is the
free group of order 2 on $\sigma = \vert X \vert$ many generators, so the
definitions above only depend on $\sigma$ up to renaming.  Given an
$\mathcal{L}_\sigma^-$-structure $H$, we will often expand it trivially to a
$\mathcal{L}_\sigma$-structure $M$ by putting a single point in $Q$.  Also, we
allow structures with empty sorts for $A$ and $J$.

For this section, we focus on $\mathcal{L}_\sigma^-$.  For $\ell=1,2$, we will
build $H_{\ell, \mathcal{D}}$ as the colimit of the
$\subset_{\mathcal{L}_\sigma^-}$-directed system $\seq{H_{\ell, d} \mid d \in
\mathcal{D}}$.

We focus first on $H_{1,\mathcal{D}}$.  For $d \in \mathcal{D}$, $H_{1, d}$ is
the substructure of $H_{1, \mathcal{D}}$ with universe $A_d = \mathcal{F} \times
\ceil{d} \times G$ and $I_d = \mathcal{F} \times \ceil{d}$.  For each of the
functions and relations below, we replace `$\mathcal{D}$' with `$d$' to denote
the restriction to $H_{1, d}$, for example $\pi_d = \pi_\mathcal{D} \rest A_d$.

\begin{definition} $H_{1, \mathcal{D}}$ is the $\mathcal{L}_\sigma^-$-structure with universe $A_\mathcal{D} = \mathcal{F} \times \mathcal{D} \times G$ and $I_\mathcal{D} = \mathcal{F} \times \mathcal{D}$ with the following functions and relations:
	\begin{itemize}
		\item $\pi_\mathcal{D}$ is the natural projection from
$A_\mathcal{D}$ to $I_\mathcal{D}$, and $E'_\mathcal{D}$ is the derived equivalence relation;
		\item $E_\mathcal{D}$ refines $E'_\mathcal{D}$ and is given by $(f, d, u) E_\mathcal{D} (f', d', u')$ iff $(f, d) = (f',d')$ and there are $d_0, \dots, d_{2n-1} \in \floor{d}$ such that $u \Delta \{f(d_0)\} \Delta\dots\Delta\{f(d_{2n-1})\} = u'$;
		\item $P_\mathcal{D}$ is the unary parity predicate and holds at $(f, d, u)$ iff $|u \cap \ran(f \rest \floor{d})|$ is odd;
		\item For $v \in G$, $D_v^\mathcal{D}$ is a unary difference predicate and holds at $(f,d,u)$ iff $u - \ran(f \rest \floor{d}) \subset v$ and $v \cap \ran f \rest \floor{d} = \emptyset$;
	\item $F^\mathcal{D}_c$ describe the transitive action of $G$ on each $E'_\mathcal{D}$-class given by $F^\mathcal{D}_c(f,d,u) = (f,d,u\Delta c)$; and
		\item $(f, d, u) R_\mathcal{D} (f', d', u')$ if and only if $f
\leq f'$ and if $e$ witnesses this, then $u = \{ i \in X \mid \exists^{odd} j \in u' . e(j) = i\}$.
	\end{itemize}
\end{definition}

The use of $E'$ is redundant given $\pi$, but makes the discussion of its
equivalence classes easier.  We also have that $E$ is redundant.

\begin{claim}
For all $d^* \in \D$ and $(f, d, u), (f', d', u') \in H_{1, d^*}$, we have $(f, d, u) E_{d^*} (f', d', u')$ if and only if the following hold:
\begin{itemize}
	\item $(f, d, u) E_{d^*}' (f', d', u')$;
	\item $P_{d^*}(f, d, u)$ if and only if $P_{d^*}(f', d', u')$; and
	\item for all $v \in G$, $D^{d^*}_v(f, d, u)$ if and only if $D^{d^*}_v(f', d', u')$.
\end{itemize}
In particular, within a particular $E'$-class, $E$-equivalence is determined by
the quantifier-free type of the singletons from $A$.
\end{claim}

\begin{proof} Clearly $(f,d,u)E(f',d',u')$ if and
only if
\begin{itemize}
\item $(f,d,u)E'(f',d',u')$ (hence $f=f'$ and $d=d'$);
\item $u \Delta u' \subset \ran(f)$ is even;
and
\item $u - \ran(f \rest \floor{d}) = u' - \ran(f' \rest \floor{d'})$
\end{itemize}
It is not hard to see that this is equivalent to the list from the
claim.\end{proof}

\begin{definition} For $d \in \D$, set $g_d$ to be the permutation on $A_d \cup I_d$ of
order two given by $g_d(f,d',u) = (f,d',u+\{f(d)\})$ on $A_d$ and the identity on $I_d$.   \end{definition}

Note that $g_d$ is defined on $(f, d', u)$ only if $d' \tri d$.  If $\neg(d'  \tri d)$, we could define $g_d(f, d', u)$, but it would not be a member of $A_d$.  This is a simply a bijection on the underlying set of $H_{1, d}$, but we can describe its interaction with the $\mathcal{L}_\sigma^-$ structure as well.

\begin{claim}\label{claim2d} Given $d_1,  d_2 \in \floor{d}$, $g_{d_1} \circ
g_{d_2}$ is an $\mathcal{L}_\sigma^-$-automorphism of $H_{1, d}$.  \end{claim}

\begin{proof}
Most of this is clear from the definition of $H_{1,d}$.  For a permutation $f$ of $M$, we say that some $f$ \emph{preserves} a predicate $U$ iff $U$ holds of $x$ iff it holds of $f(x)$ (in $M$) and it \emph{flips} $U$ iff $U$ holds of $x$ iff it fails to hold at $f(x)$.  It is easy to see that $g_{d_\ell}$ preserves each difference predicate and flips each parity predicate.  We show that $g_{d_1}$
already preserves $R_\mathcal{D}$ on any pair for which the function is defined.  Suppose
$(f, d, u) R_\mathcal{D} (f', d', u')$ and let $f = e \circ f'$.  Since $u = \{ i \in X \mid \exists^{odd} j \in u' . e(j) = i\}$, we have that 
\begin{eqnarray*}
u \Delta \{ f(d_1)\} &=& \{ i \in X \mid \exists^{odd} j \in u' . e(j) = i\} \Delta \{ e \circ f'(d_1) \}\\
 &=& \{ i \in X \mid
\exists^{odd} j \in u' \Delta \{f'(d_1) \} .  e(j) = i\}.
\end{eqnarray*}
The second equality holds because $u' \Delta \{f'(d_1)\}$ changes the number of preimages of $f(d_1)$ by 1 (when compared to $u'$). It follows that $(f, \alpha, u \Delta \{f(d_1)\}) R_\mathcal{D} (f', \alpha', u'
\Delta \{ f'(d_1)\})$.

Finally, we note that $E_d$ is preserved because the same elements
of $\D$ witnessing $E_d$-relatedness in $H_{1,d}$ will witness $E_d$ relatedness of
the $g_{d_1} \circ g_{d_2}$-images.  Since each predicate is flipped or preserved, the composition $g_{d_1} \circ g_{d_2}$ preserves each predicate and is an $\mathcal{L}_\sigma^-$-isomorphism.\end{proof}

\begin{claim} \label{claim5d}
If $\langle (f, d, u_1), (f, d, u_2) \rangle$ has the same quantifier-free type
in $H_{1, d^*}$  as $\langle (f, d, v_1), (f, d, v_2) \rangle$ for any $d \tri d^*$, then $u_1 \Delta u_2 = v_1 \Delta v_2$.
\end{claim}

\begin{proof}
Note that ``$F_{u_1 \Delta u_2} (x) = y$" is in the quantifier free type of the first and $F^{d^*}_c(g, d, u) = (g, d, v)$ if and only if $c = u \Delta v$.
\end{proof}

The $H_{2, d}$'s are built as the $g_d$-images of the $H_{1,d}$.

\begin{definition}
For $d \in \D$, set $H_{2, d}$ to be the $\mathcal{L}_\sigma^-$-structure with
universe $A_d \cup I_d$ defined so that $g_d$ is an
$\mathcal{L}_\sigma^-$-isomorphism from $H_{1,d}$.  Set $H_{2, \D}$ be the direct union of the sequence $\langle
H_{2, d} \mid d \in \D \rangle$.
\end{definition}

This definition is justified because $g_d$ is a bijection from $A_d \cup I_d$ to itself and, by Claim \ref{claim2d}, if $d \tri d' \in \D$, then $H_{2, d} \subseteq H_{2, d'}$; here, $\subseteq$ refers to the substructure relation.  Note that $H_{2, \D}$ has the same universe as $H_{1, \D}$.

Tameness type assumptions about our class of models will give an $\mathcal{L}_\sigma^-$-isomorphism $h$
from $H_{1,\D}$ to $H_{2,\D}$ with an additional property.  We call this
additional property ``respecting $\pi$".

\begin{definition}
We say that $h:H_{1, d} \to H_{2,d}$ \emph{respects} $\pi$ iff for all $a \in A_\mathcal{D}$, $\pi_\mathcal{D}(a) = \pi_\mathcal{D}\left(h(a)\right)$.
\end{definition}

In other words, we require $h \rest I_\D$ to be the identity.

From such an isomorphism we will prove
$\#(\D,\mathcal{F})$.  Note that setting $h_0(f, d,u) = (f, d, u \Delta \{f(d)\})$ is a $\mathcal{L}_\sigma^--\{R\}$-isomorphism respecting $\pi$. This $h_0$ does not preserve $R$ because different $d$ give
different $f(d)$.  This could be remedied by picking a ``generic'' or
``average'' value to play the role of $f(d)$, and we could use an ultrafilter to
find such a value.  Since we can derive an ultrafilter from the existence of such an isomorphism, this argues that this average construction is essentially the only way to construct such an isomorphism.

We say that $\mathcal{F}$ is countably closed if for any $\leq$-increasing
sequence $\seq{f_n \mid n < \omega}$ from $\mathcal{F}$, there is $f \in
\mathcal{F}$ such that $f \geq f_n$ for all $n<\omega$.

\begin{lemma}\label{tame-like} Suppose that $\mathcal{F}$ is countably closed.
If there is an $\mathcal{L}_\sigma^-$-isomorphism $h$ from $H_{1,\D}$ to
$H_{2,\D}$ respecting $\pi$, then $\#(\D,\mathcal{F})$ holds. \end{lemma}

For the remainder of the section,  we assume that $\mathcal{F}$ is countably closed and that there is an $h$ as in the lemma
to derive $\#$.

\begin{claim}\label{h} The following are true of $h$:
\begin{enumerate}
\item If we let $u_{f,d}$ be the unique element of $G$ such that $h(f,d,
\emptyset) = (f,d,u_{f,d})$, then $u_{f,d}$ doesn't depend on $d$. Hence we
denote the common value by $u_f$.
\item For all $d \in \D$, $f \in \mathcal{F}$ and $u \in G$, $h(f,d,u) = (f,d,u
\Delta u_f)$.
\item If $f \leq f'$, then $|u_{f}| \leq |u_{f'}|$.
\item For all $f \in \mathcal{F}$, $u_f \neq \emptyset$.
\item For all $f \in \mathcal{F}$, $u_f \subseteq \eran f$.

\end{enumerate}
\end{claim}

\begin{proof}
For (1), applying $h$
\[ H_{1, \D} \vDash (f, d, \emptyset) R (f, d', \emptyset) \rightarrow H_{2,
\D} \vDash (f, d, u_{f, d} ) R(f, d', u_{f, d'}).\]

Recall from the proof of Claim \ref{claim2d} that $g_d$ preserves $R$ and note that $\id$ is the witness that $f \leq f$.  Applying the definition of $R$, we have
\[u_{f, d} = \{i \in X \mid \exists^{odd} j \in u_{f, d'}. j = i\} = u_{f, d'}\]

For (2) we apply Claim \ref{claim5d} to $\langle (f, d, \emptyset), (f, d,
u)\rangle$ and $\langle (f, d, u_{f}), (f, d, v)\rangle$ where
$h(f,d,u)=(f,d,v)$.

For (3), we let $e$ be any function such that $f = e \circ f'$.  Then
$(f,d,\emptyset)R_\mathcal{D}(f',d' \emptyset)$ implies $H_{2, \mathcal{D}} \vDash ``(f, d, u_f) R (f', d, u_{f'})''$ implies $(f, d, u_f) R_\mathcal{D} (f', d, u_{f'})$ because $g_d$ preserves $R$.  So
$u_f \subseteq e`` u_{f'}$ and $|u_f| \leq |u_{f'}|$.

For (4), note that $H_{2, d}$ interprets the parity predicate to mean ``$|u|$
is even'' since $g_d$ flips $P$ and that $H_{2, \D} \vDash \neg P(f, d, u_f)$, since $H_{1,\D} \vDash
\neg P(f,d,\emptyset)$.  Thus $|u_f|$ is odd and can't be empty.

For (5), for all $d \in \D$, $H_{2, \D} \vDash D_\emptyset(f, d, u_f)$, since
$H_{1,\D} \vDash D_\emptyset(f,d,\emptyset)$.   Moreover $g_d$ preserves this predicate.  So $u_f \subset ran (f \rest \floor{d})$.\end{proof}

We are ready to produce the $f^*$ for $\#$.  It is here that we use for the
first (and only) time the countable closure of the space of functions
$\mathcal{F}$ under the order $\leq$.

\begin{claim} \label{claim8}
There is $f^* \in \mathcal{F}$ such that $|u_f| = |u_{f^*}|$ for all $f \geq f^*$ from $\mathcal{F}$.  Moreover, if $e$ witnesses $f^* \leq f$, then $e \rest u_f$ is a bijection from $u_f$ to $u_{f^*}$.
\end{claim}

\begin{proof} The moreover part follows from the first part because every member
of $u_{f^*}$ is in the image of $u_f$ under an $e$ witnessing $f^* \leq f$ by
the proof of Claim \ref{h} part (3).

Suppose there is no such $f^*$.  Then there is a $\leq$-increasing sequence
$\langle f_n \in \mathcal{F} \mid n <\omega \rangle$ such that $|u_{f_n}| <
|u_{f_{n+1}}|$ for all $n<\omega$.  By our assumption that $\mathcal{F}$ is
countably closed we can find $f^* \geq f_n$ for all $n<\omega$, but then
$|u_{f^*}|$ is a natural number above infinitely many natural numbers, a
contradiction.  \end{proof}




So we have derived $\#(\D,\mathcal{F})$ as witnessed by $f^*$ and the finite
sets $u_f$.  This finishes the construction of our sequence of models.  We will
need further work to show that these models can be thought of as elements of
some AEC and that tameness assumptions about that AEC give the hypothesis of
Lemma \ref{tame-like}.

\section{Abstract elementary classes} \label{phase3}

The goal of this section is twofold. First, we put the algebraic constructions of Section \ref{phase1} into the context of AECs.  Second, we put the necessary pieces together to conclude large cardinal principles from global tameness and locality axioms.

The AEC is designed to precisely take in the algebraic examples constructed
in Section \ref{phase1} with a single twist.  Recall $\mathcal{L}_\sigma$ from Definition \ref{language-def}.  The extra predicate $Q$  is an index for copies of the algebraic construction, as in Baldwin and Shelah \cite{nonlocality}.  This allows us to turn the ``incompactness'' results about the existence of isomorphisms above into the desired nonlocality results for Galois types.


We define an AEC parameterized by $\sigma$ with very minimal structure.  In
applications we require $\sigma^\omega = \sigma$.  The strong substructure
relation is as weak as possible, leaving open the question of whether
restricting to the case of stronger strong substructure relations carries the
same large cardinal implications.

\begin{definition}\label{ksig-def} We define $\mathbb{K} = \mathbb{K_\sigma}$ to be the collection of $\mathcal{L}_\sigma$-structures (recall Definition \ref{language-def}) given by $M \in \mathbb{K}$ if and only if $M$ is an $\mathcal{L}_\sigma$-structure
satisfying:
\begin{enumerate}
\item $\{F_c \mid c \in G\}$ is an 1-transitive action of $G$ on $\pi^{-1}\{i\}$ for every $i \in I$.
\item $E'$ and $E$ are equivalence relations on $A$ and $aE'b$ if and only if
$\pi(a)=\pi(b)$.
\item For all $i \in I$ and $j \in J$, there is an $a \in A$ such that
$\pi(a)=i$ and $Q(a)=j$.
\end{enumerate}
We let $\prec_{\mathbb{K}}$ be the $\mathcal{L}_\sigma$-substructure relation. \end{definition}

Note that this is an AEC with $LS(\mathbb{K}) = \sigma$; in fact, $\mathbb{K}_\sigma$ is the class of models of an $L_{\sigma^+, \omega}$-sentence.  The main difference
between this definition and Shelah \cite[$\boxtimes_2$ in Proof of Theorem
1.3]{shelah932} is that we have encoded the entire group $G$ into the language
rather than adding a separate sort for it (see Remark
\ref{shelahfix}).  We note that the structures
we call $H_{\ell, d}$ and $M_{\ell, d}$ are called $M_{\ell, \alpha}$ and
$M^+_{\ell, \alpha}$ respectively in \cite{shelah932}.

As mentioned in the last section, any $\mathcal{L}_\sigma^-$-structure can be
trivially expanded to an $\mathcal{L}_\sigma$-structure by putting a single point in $J$
and fixing $Q$ to be the constant function with this value on $A$.  For $\ell= 1,2$, let
$M_{\ell, d}$ be this expansion of $H_{\ell, d}$ and name the single element of
$J^{H_{\ell,d}}$ as $i_\ell$.  We similarly expand $H_{\ell, \D}$ to $M_{\ell,
\D}$.

We also define $M_{0, d}$ to be the $\mathcal{L}_\sigma$-structure with $A$ and $J$ empty and $I = I_d$; note that $M_{0,d} \in \mathbb{K}_\sigma$.  We similarly define $M_{0, \D}$.

Note that $M_{0, d} \subseteq M_{\ell, d}$ for $\ell = 1, 2$.  Thus we can define $p_d:= gtp(i_1/M_{0, d}; M_{1, d})$ and $q_d := gtp(i_2/M_{0,d}; M_{2, d})$; $p_\D$ and $q_\D$ are defined similarly.  The connection between this AEC and the previous work is the following proposition.

\begin{claim} \label{typeeq-claim} \
\begin{enumerate}
	\item For all $d\in \D$, $p_d = q_d$.
	\item There is an isomorphism as in Lemma \ref{tame-like} if and only if $p_\D = q_\D$.
\end{enumerate}
\end{claim}

To prove the forward direction of (2), we need the notion of \emph{admitting intersection} coming from \cite[Definition 1.2]{nonlocality} in the AEC case.

\begin{definition} \label{admits}
$\mathbb{K}$ admits intersections if and only if for all $X \subseteq M \in \mathbb{K}$, $\cl_M(X) \prec_\mathbb{K} M$, where $\cl_M(X)$ is the substructure of $M$ with universe $\cap \{ N : X \subseteq N \prec_\mathbb{K} M\}$.
\end{definition}

The key consequence of closure under intersection is that it simplifies checking if two types are equal.

\begin{fact}[\cite{nonlocality}.1.3] \label{typeequality-fact}
Suppose $\mathbb{K}$ admits intersections.  Then $gtp(a_1/M_0; M_1) =
gtp(a_2/M_0; M_2)$ if and only if there is $h: \cl_{M_1}(M_0a_1) \cong_{M_0} \cl_{M_2}(M_0a_2)$ with $h(a_1) = a_2$.\end{fact}

\begin{claim}
$\mathbb{K}_\sigma$ admits intersections.
\end{claim}

\begin{proof} We define the closure on each of the predicates.  Then $\cl_M(X)$ will be the substructure with the union of the $\cl^i_M(X)$ as the universe.
\begin{itemize}
	\item $\cl^1_M(X) = \left( X \cap J \right) \cup\{ j \in J^M : \exists a \in X \cap A^M . Q^M(a) = j\}$;
	\item $\cl^2_M(X) = \left( X \cap I \right) \cup \{i \in I^M : \exists a \in X \cap A^M . \pi^M(a) = i\}$;
	\item $\cl^3_M(X) = \{ a \in A^M : \exists (i, j) \in \cl^1_M(X) \times \cl^2_M(X) . Q^M(a) = j \text{ and } \pi^M(a) = i\}$.
\end{itemize}
It is routine to verify that $\cl_M$ satisfies Definition \ref{admits}. \end{proof}

\begin{remark} \label{shelahfix}
The AEC as constructed in \cite{shelah932} is not closed under intersections.  Shelah does not require that the entire group $G$ be included in every model.  This means that if there is an empty $E'$-equivalence class of $A$ that must be filled (due to it projecting into $I$ and $J$), then any proper subgroups $G' < G$ allows a choice of orbits to fill the equivalence class.  This choice is incompatible with closure under intersection.  However, an argument similar to Claim \ref{ap-claim} still shows it has amalgamation.
\end{remark}

\begin{proof}[Proof of Claim \ref{typeeq-claim}]
We begin by showing (1) and the forward direction of (2).  The type equality
comes from the fact that if $f$ is an $\mathcal{L}_\sigma^-$-isomorphism from
$H_{1, d}$ to $H_{2, d}$ that respects $\pi$, then $f^* := f  \cup \{ (i_1, i_2)
\}$ is a $\mathcal{L}_\sigma$-isomorphism from $M_{1, d}$ to $M_{2, d}$ that
fixes $M_{0, d}$ and sends $i_1$ to $i_2$.  Since each $g_d$ respects $\pi$, this witnesses $p_d = q_d$. The same
argument gives the forward direction of (2).

For the other direction, suppose $p_\D = q_\D$.  It is easy to compute that
$\cl_{M_{\ell, \D}}(M_{0, \D} i_\ell) = M_{\ell, \D}$.  So by Fact
\ref{typeequality-fact} we have an isomorphism $h: M_{1, \D} \cong_{M_{0, \D}}
M_{2, \D}$.  This restricts to an isomorphism from $H_{1, \D}$ to $H_{2, \D}$
that respects $\pi$ as in Lemma \ref{tame-like} and so $\#(\D,\mathcal{F})$
follows.\end{proof}

Much of the work on AECs takes place under the assumption of amalgamation.  Although not necessary for this proof, we also point out that $\mathbb{K}_\sigma$ has amalgamation.  This means the use of the construction \cite[Definition 4.5]{nonlocality} in \cite{shelah932} is unnecessary.

\begin{claim}\label{ap-claim}
$\mathbb{K}_\sigma$ has amalgamation.
\end{claim}

\begin{proof} Suppose that $M_0 \subset M_1, M_2 \in K$ and without loss of
generality $M_1 \cap M_2 = M_0$.  We will define the amalgam to essentially be the disjoint union of $M_1$ and $M_2$ over $M_0$ written in the standard way with $A$ equal to $I \times G \times J$.  Thus, we define the universe of $M^*$ as follows:
\begin{itemize}
	\item $J^* = J_1 \cup J_2$;
	\item $I^* = I_1 \cup I_2$; and
	\item $A^* =  I^* \times G \times J^*$.
\end{itemize}
For each $(i, j) \in I_\ell \times J_\ell$ and $\ell = 1, 2$, pick some $x^\ell_{i, j} \in A^{M_\ell}$ such that
\begin{itemize}
	\item $\pi^{M_\ell}(x^\ell_{i,j}) = i$;
	\item $Q^{M_\ell}(x^\ell_{i,j}) = j$; and
	\item if $(i, j) \in I_0 \times J_0$, then $x^1_{i, j} = x^2_{i, j}$.
\end{itemize}
The $x^\ell_{i,j}$ serves as the ``zero'' to define the action of $G$ on $A^*$.  We define $f_\ell:M_\ell \to M^*$ as the identity on $J_\ell, I_\ell,$ and $G$ and, given $y \in A^{M_\ell}$,
$$f_\ell(y) = (i, g, j) \iff F_g^{M_\ell}(x^\ell_{i, j}) = y; \pi^{M_\ell}(y) = i;\text{ and } Q^{M_\ell}(y) = j$$
We put the $\mathcal{L}_\sigma$-structure on $A^*$ only as required by $M_1$ and $M_2$.  For instance, $E^*$ holds of $(i, g, j)$ and $(i', g', j')$ iff they are images of $f_\ell$ and their preimages are $E^{M_\ell}$ related.  Then $M^* \in \mathbb{K}_\sigma$ and is the amalgam.
\end{proof}

We are now able to put the pieces together and generate several equivalences between global tameness principles for AECs and large cardinal axioms.

\begin{definition}[\cite{baldwinbook}, Chapter 11]
Let $\mathcal{K}$ be an AEC and $\kappa \leq \lambda$.
\begin{enumerate}
	\item $\mathcal{K}$ is $(<\kappa, \lambda)$-tame if for every $M \in \mathcal{K}_\lambda$ and $p \neq q \in \gS(M)$, there is a $M_0 \prec_{\mathcal{K}} M$ of size $< \kappa$ such that $p \rest M_0 \neq q \rest M_0$.
	\item $\mathcal{K}$ is $<\kappa$-tame if it is $(<\kappa, \mu)$-tame for all $\mu \geq \kappa$.
	\item $\mathcal{K}$ is eventually tame if it is $<\kappa$-tame for some $\kappa > \LS(\mathcal{K})$.
	\item $\mathcal{K}$ is $\kappa$-local if for every $M \in \mathcal{K}$, $p \neq q \in \gS(M)$, and resolution $\langle M_i \in \mathcal{K} \mid i < \kappa\rangle$ of $M$, there is $i_0 < \kappa$ such that $p \rest M_{i_0} \neq q \rest M_{i_0}$.
\end{enumerate}
\end{definition}

\begin{theorem} \label{almost-mainthm} Let $\sigma < \kappa$ be infinite
cardinals with $\sigma^\omega=\sigma$.
\begin{enumerate}

\item If $\kappa^\sigma=\kappa$ and every AEC $\mathbb{K}$ with
$\LS(\mathbb{K})= \sigma$ is $(<\kappa,\kappa)$-tame, then $\kappa$ is
$\sigma^+$-weakly compact.

\item If every AEC $\mathbb{K}$ with $\LS(\mathbb{K})=\sigma$ is $\kappa$-local,
then $\kappa$ is $\sigma^+$-measurable.

\item If every AEC $\mathbb{K}$ with $LS(\mathbb{K}) = \sigma$ is
$(<\kappa, \sigma^{(\lambda^{<\kappa})})$-tame, then  $\kappa$ is $(\sigma^+, \lambda)$-strongly compact.

\end{enumerate}
\end{theorem}

So we have the following corollary.

\begin{corollary} \label{mainthm-cor} Let $\kappa$ be an infinite cardinal such that $\mu^\omega <
\kappa$ for all $\mu < \kappa$.
\begin{enumerate}
\item if $\kappa^{<\kappa}=\kappa$ and every AEC $\mathbb{K}$ with
$\LS(\mathbb{K})<\kappa$ is $(<\kappa,\kappa)$-tame, then $\kappa$ is almost
weakly compact.
\item if every AEC $\mathbb{K}$ with $\LS(\mathbb{K})<\kappa$ is $<\kappa$-tame, then $\kappa$ is almost strongly compact.
\end{enumerate}
\end{corollary}

\begin{proof}[Proof of Theorem \ref{almost-mainthm}] We start with the proof of
part (1).  We wish to apply our AEC construction together with Remark
\ref{rem-almost-wc}.  Construct $\mathbb{K}_\sigma$ as in Definition \ref{ksig-def}; by assumption, this is $(<\kappa, \kappa)$-tame. Let $\mathcal{A}$ be a field of subsets of $\kappa$ with
$\card{\mathcal{A}}=\kappa$.  Let $\theta$ be a big regular cardinal and take
$X \prec H_\theta$ of size $\kappa$ with $\mathcal{A} \in X$.  By our cardinal
arithmetic assumption we can take $X$ to be closed under $\sigma$-sequences.
Let $\mathcal{F}$ be the collection of functions in $X$ with domain $\kappa$ and
range a bounded subset of $\kappa$ and $(\D,\tri) = (\kappa,\in)$.  A
straightforward argument using the fact that $X$ is closed under
$\sigma$-sequences shows that $\mathcal{F}$ is $\sigma^+$-replete.  It is also
clear that $\mathcal{F}$ has a characteristic function for each $A \in
\mathcal{A}$.  Build the Galois types $\{p_\alpha, q_\alpha \mid \alpha< \kappa\}$ corresponding to this system. By the $(<\kappa, \kappa)$-tameness of $\mathbb{K}_\sigma$, $p_\kappa = q_\kappa$. By Claim \ref{typeeq-claim} and Lemma
\ref{tame-like}, we have that $\#(\D,\mathcal{F})$ holds.  Note that Lemma
\ref{tame-like} requires that $\mathcal{F}$ is countably closed and this follows
from the fact that $\sigma^\omega=\sigma$.  So by Remark \ref{rem-almost-wc}
$\kappa$ is $\sigma^+$-weakly compact.

Part (2) is essentially Shelah's theorem from \cite{shelah932}, but with the
required corrections.  We let $\mathcal{F} = {}^\kappa\sigma$ and
$(\D,\tri)=(\kappa,\in)$.  By our locality assumption, Claim \ref{typeeq-claim}
and Lemma \ref{tame-like}, we have $\#(\D,\mathcal{F})$.  Hence by Corollary
\ref{cor-almost-meas}, we have that $\kappa$ is $\sigma^+$-measurable.

For part (3) we let $(\D,\tri)=(\mathcal{P}_\kappa(\lambda), \subset)$ and
$\mathcal{F} = {}^\D\sigma$.  By our tameness assumption, Claim
\ref{typeeq-claim} and Lemma \ref{tame-like}, we have $\#(\D,\mathcal{F})$.
Hence by Remark \ref{rem-almost-sc}, $\kappa$ is $(\sigma^+,\lambda)$-strongly
compact.  \end{proof}

By strengthening the hypotheses a little we can remove the `almost' from the
above theorem.  To do so we need a definition that generalizes \cite[Definition 2.10]{tamelc}. 

\begin{definition}
$(\mathbb{K}, \prec_{\mathbb{K}})$ is quasi-essentially below $\kappa$ if and
only if $\LS(\mathbb{K}) < \kappa$ or there is a theory $T$ in $L_{\kappa, \omega}$ such that $\mathbb{K} = \text{Mod }T$ and $\prec_{\mathbb{K}}$ is implied by $\prec_{L_{\kappa, \omega}}$.
\end{definition}

We have introduced quasi-essentially below instead of just essentially below
from \cite{tamelc}, because although the class of models in $\mathbb{K}_\sigma$
are axiomatizable in $L_{\sigma^+, \omega}$, the strong substructure relation is
even weaker than first-order elementary.

\begin{theorem}\label{wc-mainthm}  Let $\kappa$ be an infinite cardinal with
$\kappa^{<\kappa}=\kappa$ and for every $\mu<\kappa$, $\mu^\omega < \kappa$.  If
every AEC $\mathbb{K}$ which is quasi-essentially below $\kappa$ is
$(<\kappa,\kappa)$-tame, then $\kappa$ is weakly compact.  \end{theorem}

\begin{proof} The proof follows the proof of Theorem \ref{almost-mainthm}.(1),
and we point out the differences.  The additional cardinal arithmetic implies
that $X \prec H_\theta$ can be taken to closed under $<\kappa$-sequences.  Then set $\mathcal{F}$ to be the collection of functions
in $X$ with domain $\kappa$ and range bounded in $\kappa$.  Note that
$\mathcal{F}$ is countably closed since $\kappa$ has uncountable cofinality and
${}^\kappa\sigma$ is countably closed provided that $\sigma^\omega=\sigma$.  As before, we code this into $\mathbb{K}_\kappa$ to conclude that $\#(\D, \mathcal{F})$ holds.  Moreover $\mathcal{F}$ is $\kappa$-replete by the
closure of $X$, hence by Corollary \ref{cor-wc}, $\kappa$ is weakly
compact. \end{proof}

\begin{theorem}\label{sc-mainthm} Let $\kappa$ be a cardinal with
$\cf(\kappa)>\omega$ and for all $\mu<\kappa$, $\mu^\omega<\kappa$.  If every
AEC $\mathbb{K}$ that is quasi-essentially below $\kappa$ is $<\kappa$-tame, then $\kappa$ is strongly compact.  \end{theorem}

This theorem has some level by level information.  In particular
$(<\kappa, \sup_{\alpha < \kappa} \alpha^{(\lambda^{<\kappa})})$-tameness will
give that $\kappa$ is $\lambda$-strongly compact.

\begin{proof} The proof is similar to the other proofs above.  We take
$(\D,\tri)=(\mathcal{P}_\kappa(\lambda),\subset)$ and $\mathcal{F}$ to be the
set of functions with domain $\D$ and range bounded in $\kappa$.  A similar
argument to the one in the previous theorem shows that $\mathcal{F}$ is
countably complete.  Our tameness assumption gives $\#(\D,\mathcal{F})$ and
hence $\kappa$ is $\lambda$-strongly compact by Corollary \ref{cor-sc}.
\end{proof}

It is important to note that the converses of the main theorems (and
corollaries) from this section are true.  The bulk of the work is already done
in the first author's paper \cite{tamelc}.  In some cases the converses of the
stated results are stronger than theorems appearing in the literature.  With
some small adjustment the proofs in the literature already give these stronger
claims.  We make a brief list of the improvements required.

\begin{enumerate}
\item \L o\'{s}' Theorem for AECs \cite[Theorem 4.3]{tamelc} holds for AECs
with L\"{o}wenheim-Skolem number $\sigma$ and any $\sigma^+$-complete
ultrafilter.  This allows us to prove for example that every AEC $\mathbb{K}$
with $\LS(\mathbb{K})<\kappa$ is $<\kappa$-tame from $\kappa$ is almost strongly
compact.

\item \L o\'{s}' Theorem for AECs does not require an ultrafilter only that the
filter measures enough sets.  This allows us to prove every AEC $\mathbb{K}$
with $\LS(\mathbb{K})<\kappa$ is $(<\kappa,\kappa)$-tame from the assumption
that $\kappa$ is almost weakly compact.  In particular we can build everything
into a transitive model of set theory of size $\kappa$ and the weak compactness
assumption gives a filter measuring all subsets of $\kappa$ in the model.  This
is enough to complete the proof.

\item \L o\'{s}' Theorem for AECs applies to the class of AECs which are
quasi-essentially below $\kappa$.  This allows us to prove that every AEC
$\mathbb{K}$ which is quasi-essentially below $\kappa$ is $<\kappa$-tame from
$\kappa$ is strongly compact.  It could be that such AECs have no models of size
less than $\kappa$, but existing arguments are enough to give tameness over sets\footnote{Typically, Galois types are defined so that the domains are always models.  The same definition works for defining Galois types over arbitrary sets.  However, many model-theoretic arguments (\cite[Claim 3.3]{sh394} on local character of non-splitting from stability is an early example) only work for Galois types over models, explaining their prevalence.  The set-theoretic nature of the arguments from large cardinals, on the other hand, mean that they carry through with little change.}
rather than models.
\end{enumerate}

We collect a few remarks on our construction:
\begin{enumerate}
\item We do not know if the cardinal arithmetic assumptions are necessary in the
main theorems of this section.  For example, if we assume that $\kappa$ is
weakly compact and we force to add $\kappa^+$ many subsets to some $\sigma^+$
where $\sigma<\kappa$, then $\kappa$ remains $\sigma^+$-weakly compact in the
extension.  It follows that every AEC with L\"{o}wenheim-Skolem number $\sigma$
is $(<\kappa,\kappa)$-tame in the extension.  We do not know if $\kappa$
satisfies any stronger tameness properties in the extension.

\item Under our mild cardinal arithmetic assumptions, the global full tameness
and type shortness and compactness results from \cite{tamelc} follow from the
global tameness for 1-types, as this tameness is already enough to imply the
necessary large cardinals.
\end{enumerate}

We conclude this section with an application to category theory.  There has been
recent activity in exploring the connection between AECs and accessible
categories; see Lieberman \cite{lieberman}, Beke and Rosicky \cite{bekerosicky},
and Lieberman and Rosicky \cite{liebermanrosicky}.  In \cite[Theorem
5.2]{liebermanrosicky}, the authors apply a result of Makkai and Pare to derive a global version of Fact \ref{boneythm} from class many strong compacts.  Here we show that this application is in fact equivalent to
the whole result.  Note that in the global version we do not need any cardinal
arithmetic assumptions.

\begin{corollary} \label{ct-cor}
The following are equivalent:
\begin{enumerate}
	\item The powerful image of any accessible functor is accessible.
	\item Every AEC is eventually tame.
	\item There are class many almost strongly compact cardinals.
\end{enumerate}

Fix and infinite cardinal $\kappa$ with $\mu^\omega < \kappa$ for all $\mu < \kappa$.  The following are equivalent:
\begin{enumerate}
	\item The powerful image of a $<\kappa$-accessible functor is $<\kappa$-accessible.
	\item Every AEC with $LS(K) < \kappa$ is $<\kappa$-tame.
	\item $\kappa$ is almost strongly compact.
\end{enumerate}
\end{corollary}

In saying that every AEC is eventually tame, we allow AECs with no models of size $\kappa$ or larger to be trivially tame by saying they are $<\kappa$-tame.  For the category theoretic notions in this corollary, see \cite{makkaipare}.  In particular, given a functor $F:\mathcal{K} \to \mathcal{L}$, the powerful image of $F$ is the subcategory of $\mathcal{L}$ whose objects are $Fx$ for $x \in \mathcal{K}$ and whose arrows are any arrow from $\mathcal{L}$ between these objects.

\begin{proof} In the first set of equivalences, the first implies the second by \cite[Theorem 5.2]{liebermanrosicky}.  The third implies the first by Brooke-Taylor and Rosicky \cite[Corollary 3.5]{brooketaylorrosicky}, which is a modification of Makkai and Pare's original \cite[5.5.1]{makkaipare}.\\  
To see the second implies the third, for all $\sigma$, we know that there is some $\kappa_\sigma$ such that $\mathbb{K}_\sigma$ is $<\kappa_\sigma$-tame.  Let $\mathbf{S}$ be the set of all limit points of the map that takes $\sigma$ to $\sigma^\omega + \kappa_\sigma$.  Clearly, $\mathbf{S}$ is class sized.

We claim that each $\kappa \in \mathbf{S}$ is almost strongly compact.  First note that $\sigma^\omega < \kappa$ for all $\sigma < \kappa$.  Let $\sigma < \kappa \leq \lambda$.  Then $\mathbb{K}_{\sigma^\omega}$ is $<\kappa$-tame, so it's $\left(<\kappa, (\sigma^\omega)^{(\lambda^{<\kappa})}\right)$-tame.  The proof of Theorem \ref{almost-mainthm}.(3) only involves this AEC, so it implies that $\kappa$ is $\left((\sigma^\omega)^+, \lambda\right)$-strongly compact.  Of course, this means that it is $(\sigma, \lambda)$-strongly compact.  Since $\sigma$ and $\lambda$ were arbitrary, $\kappa$ is almost strongly compact, as desired.

The second set of equivalences is just the parameterized version of the first
one, and follows by the parameterized versions of the relevant results.  The
cardinal arithmetic is only needed for $(2)$ implies $(3)$. \end{proof}

\section{The consistency strength of
$(<\kappa,\kappa)$-tameness}\label{strength}

We have already remarked that we do not know if the cardinal arithmetic assumptions are necessary in the theorems of Section 4.  In this section, we show that in the absence of cardinal arithmetic assumptions, the degree of tameness we associate to weak compactness has the expected \emph{consistency strength}.

\begin{theorem} Let $\kappa$ be a regular cardinal greater than $\aleph_1$. If
every AEC $\mathbb{K}$ which is quasi-essentially below $\kappa$ is
$(<\kappa,\kappa)$-tame, then $\kappa$ is weakly compact in $L$. \end{theorem}

\begin{proof} We may assume that $0^\#$ does not exist, since otherwise every uncountable cardinal is weakly compact in $L$ (see \cite[Theorems 9.17.(b) and 9.14.(b)]{kanamori}).  Let $\mathcal{A}$ in $L$ be
a collection of $\kappa$ many subsets of $\kappa$ which is closed under
complements and intersections of size less than $\kappa$.  Choose an ordinal
$\beta<(\kappa^+)^L$ with $\mathcal{A} \in L_\beta$ and such that $L$ models
$[L_\beta \cap \mathcal{P}(\kappa)]^{<\kappa} \subseteq L_\beta$.  Let
$\mathcal{F}$ be the collection of functions in $L_\beta$ from $\kappa$ to
$\kappa$ whose ranges are bounded in $\kappa$.

We claim that $\mathcal{F}$ is countably closed in $V$ under the ordering on
functions defined in Section \ref{phase2}.  Suppose that $X$ is a countable
subset of $\mathcal{F}$.  By the covering lemma and our assumption that $0^\#$
doesn't exist, there is a set $Y \in L$ of size $\aleph_1$ with $X \subseteq Y$.
By the choice of $L_\beta$, $Y \in L_\beta$ and hence it has an upperbound in
$L_\beta$.

By our tameness assumption, Claim \ref{typeeq-claim} and Lemma \ref{tame-like},
we have $\#(\kappa,\mathcal{F})$ from which we can derive a filter $U$ on
$\kappa$.  From the way we chose $\mathcal{F}$, $U$ measures all subsets of
$\kappa$ in $L_\beta$ and is $\kappa$-complete with respect to sequences in
$L_\beta$.  We are now ready to give a standard argument that $U$ restricted to
$\mathcal{A}$ is in $L$.

Let $j:L_\beta \to L_\gamma \simeq \Ult(L_\beta,U)$ be the elementary embedding
derived from the ultrapower by $U$.  Standard arguments show that the critical
point of $j$ is $\kappa$.  Let $\langle A_\alpha \mid \alpha < \kappa \rangle$
be an enumeration of $\mathcal{A}$ in $L_\beta$.  The sequence $\langle
j(A_\alpha) \mid \alpha < \kappa \rangle$ is in $L$, since it is just $j(\langle
A_\alpha \mid \alpha < \kappa \rangle) \upharpoonright \kappa$.  So the set
$\bar{U} = \{ A_\alpha \mid \kappa \in j(A_\alpha) \}$ is in $L$.  It is easy to
see that $\bar{U} \subseteq U$ and hence is the $\kappa$-complete
$\mathcal{A}$-ultrafilter that we require. \end{proof}

We expect that a similar result can be proved for locality and almost measurability with an appropriate inner model in place of $L$.  Of course, extending this result to the almost strongly compact case would require major advances in inner model theory.
\bibliographystyle{amsplain}
\bibliography{references}
\end{document}